\documentclass[preprint,12pt]{elsarticle}

\usepackage{hyperref}
\usepackage{xcolor}
\usepackage{amssymb}
\usepackage{amsmath}
\usepackage{amsthm}

\usepackage[english]{babel}
\usepackage{graphicx,epstopdf,epsfig}
\usepackage{amsfonts,epsfig,fancyhdr,graphics}
\newtheorem{theorem}{Theorem}[section]
\newtheorem{lemma}[theorem]{Lemma}
\newtheorem{corollary}[theorem]{Corollary}
\newtheorem{proposition}[theorem]{Proposition}

\usepackage{xcolor}
\newtheorem{example}[theorem]{Example}
\newdefinition{remark}{Remark}

\def\tr{\text{tr}}
\def\Re{\text{Re}}

\journal{Linear Algebra and its Applications}

\begin{document}

\begin{frontmatter}



\title{Zero-trace index of matrices}


\author{Kennett L. Dela Rosa, Eloise P. Misa} 

\affiliation{organization={Institute of Mathematics, University of the Philippines Diliman},
            addressline={C.P. Garcia Ave., UP Campus}, 
            city={Quezon City},
            postcode={1101}, 
            country={Philippines}}

\begin{abstract}
The zero-trace index $t(A)$ of a square complex matrix $A$ is introduced. 
It is defined as the largest integer $k$ such that $A$ admits a $k \times k$ compression with trace zero. 
This concept is analogous to the zero-dilation index $d(A)$ and is closely related to the $k$-numerical range. 
The $k$-numerical range is characterized as an intersection of half-planes, which is used to derive a formula for $t(A)$ that reduces its computation to the Hermitian case. 
Conditions for $t(A) = d(A)$ are also established, and matrices satisfying $t(A) = n - 1$ are characterized. Finally, the zero-trace index is determined for several classes of matrices.
\end{abstract}




\end{frontmatter}



\section{Introduction}
	Let $M_{n,m}$ be the set of all $n\times m$ complex matrices, and use $M_n$ if $n=m$. 
	Let $I_n\in M_n$ denote the identity matrix.
	For $A\in M_n$ and $1\leq k\leq n$, a complex matrix $B\in M_k$ is said to \emph{dilate} to $A$ or $B$ is a \emph{compression} of $A$ if $B=P^*AP$ for some matrix $P\in M_{n,k}$ such that $P^*P=I_k$ or equivalently,
	if $A$ is unitarily similar to a matrix of the form $\begin{bmatrix} B&\ast\\ \ast&\ast \end{bmatrix}$. 
	The \emph{zero-dilation index} $d(A)$ of $A$ is the largest integer $k$ for which the zero matrix $0_k\in M_k$ dilates to $A$. 
	The zero-dilation index was introduced in \cite{GauWangWu2014} and is closely related to a generalization of the classical numerical range called the higher-rank numerical ranges. 
	For $1\leq k\leq n$, the \emph{rank-$k$ numerical range} of $A$ is given by $\Lambda_k(A)=\{\lambda\in\mathbb{C} : \lambda I_k \text{ dilates to } A\}$. 
	Hence, the zero-dilation index of $A$ is the largest $k$ for which $0 \in \Lambda_k(A)$. 
	The higher-rank numerical range was first defined in the context of quantum error correction \cite{ChoiKribsZyczkowski2006, ChoiKribsZyczkowski2006b}. When $k=1$, $\Lambda_1(A)$ is equal to the classical \emph{numerical range} $W(A)$, defined as
	\(
	W(A) = \{ x^* A x : x \in \mathbb{C}^n, \ x^*x = 1 \}
	\).
	It is known that the $\Lambda_k(A)'s$ are always convex \cite{LiSze2008, Woerdeman2008}.
	In particular, Li and Sze in \cite{LiSze2008} showed that the rank-$k$ numerical ranges can be expressed as an intersection of half-planes  
	\begin{equation} \label{rank k int}
	\Lambda_k(A) = \bigcap_{\theta \in [0,2\pi)} \left\{ \lambda \in \mathbb{C} : \mathrm{Re}(e^{i\theta} \lambda) \leq \lambda_k(\mathrm{Re}(e^{i\theta} A)) \right\},
	\end{equation} 
	where for $B\in M_n$, $\Re(B)=\frac{B+B^*}{2}$ denotes the \emph{real part} of $B$, and $\lambda_k(C)$ denotes the $k$th largest eigenvalue of a Hermitian matrix $C\in M_n$.
	With this description, $d(A)$ can be expressed as 
	\begin{equation}\label{d(A) as min}
	\min\left\{k : 0\leq \lambda_k(\Re(e^{i\theta}A)) , \theta\in [0,2\pi)\right\}
	\end{equation}
	(see \cite[Theorem 3.1]{LiSze2008}).
	For a Hermitian matrix $C$, let $i_{\geq 0}(C)$ denote the number of nonnegative eigenvalues of $C$ counting multiplicities. 
	Then (\ref{d(A) as min}) can be written as
	\[
	d(A) = \min \left\{ i_{\geq 0}(\text{Re}(e^{i\theta} A)) : \theta \in \mathbb{R} \right\}
	\]
	\cite[Theorem 2.2]{GauWangWu2014}. 
	For Hermitian $C\in M_n$, $d(C)=\min\{i_{\geq 0}(C),i_{\geq 0}(-C)\}$ \cite[Corollary 2.3]{GauWangWu2014}. 
	Furthermore, $d(A) \leq \left\lceil \frac{n}{2} \right\rceil$ whenever $\dim\ker(\Re(e^{i\theta} A)) \leq 1$ for some $\theta \in \mathbb{R}$ \cite[Corollary 2.5]{GauWangWu2014}.  
	Matrices $A\in M_n$ that satisfy $d(A) = n - 1$ were characterized in \cite[Theorem 3.3]{GauWangWu2014}. 
	In particular, row stochastic matrices and doubly stochastic matrices $A$ with $d(A)=n-1$ were  described in \cite{GauWangWu2016}. 
	The zero-dilation index has been determined for the following classes of matrices: normal matrices, weighted permutation matrices, $S_n$-matrices, companion matrices, KMS matrices, and the block matrix analogues of companion and  KMS matrices \cite{GauWangWu2014, GauWu2014a, GauWu2016, DelaRosa2025}.\\

	In this paper, we define the \emph{zero-trace index} of $A$, denoted $t(A)$, as
	\[
	t(A)=\text{max}\left\{k : \tr(P^*AP)=0 \text{ for some } P\in M_{n,k} \text{ such that } P^*P=I_k \right\}.	
	\]
	That is, the zero-trace index is the largest size of a compression of $A$ whose trace is $0$. Observe that if the zero matrix $0_k$ is a compression of $A$, then $t(A)\geq k$. Hence, $d(A)\leq t(A)$. 
	Similar to the zero-dilation index, the zero-trace index is related to another generalization of the classical numerical range called the $k$-numerical ranges. 
	For $1\leq k\leq n$, the \emph{$k$-numerical range} of $A$ is given by 
	\[
	W_k(A)= \left\{\tr(P^*AP) : P\in M_{n,k}\text{ such that } P^*P=I_k\right\}.
	\]
	Notice that $z\in W_k(A)$ if and only if $z=\sum_{j=1}^k x_j^*Ax_j$ for some orthonormal subset $\{x_1,\dots,x_k\}$ of $\mathbb{C}^n$.
	The zero-trace index can be expressed as the largest integer $k$ such that $0\in W_k(A)$.
	The $k$-numerical range of $A$ was first introduced by Halmos in \citep{Halmos1964} and its convexity was proved by Berger in \citep{Berger1963} (see \cite[Problem 211]{Halmos1982} for a simpler proof). 
	When $k=1$, $W_1(A)$ reduces to the classical numerical range $W(A)$. 
	We follow \cite{MarcusFilippenko1978} and use the non-normalized definition of the $k$-numerical range.\\

	In Section 2, we show that the $k$-numerical range is the intersection of half-planes determined by sums of the eigenvalues of $\Re(e^{i\theta}A)$.  
	Specifically, we show that 
	\begin{equation}\label{Wk int eqn}
	W_k(A) = \bigcap_{\theta \in [0,2\pi)} \left\{ \lambda \in \mathbb{C} : \Re(e^{i\theta} \lambda) \leq \sum_{j=1}^k \lambda_j(\Re(e^{i\theta} A)) \right\}.
	\end{equation}
	In Section 3, we establish fundamental properties of the zero-trace index. We show that $t(A)=d(A)$ when $0$ is an extreme point of $W(A)$ (Theorem \ref{0 extreme point}) and characterize Hermitian matrices $A\in M_3$ that satisfy the equality (Proposition \ref{3x3 hermitian}). We give necessary and sufficient conditions for $t(A)=n-1$ (Theorem \ref{t n-1}) and $t(A\oplus A)=2t(A)$ (Proposition \ref{t(A plus A)}), and show that $t(A\oplus 0_m)=t(A)+m$ (Proposition \ref{t(a+0)}). 
	 In Section 4, we use (\ref{Wk int eqn}) to derive an alternative formula for $t(A)$ (Theorem \ref{t index min}) that is analogous to the description for the zero-dilation index as a minimum value (\cite[Theorem 2.2]{GauWangWu2014}). This relates the compression-based definition of the zero-trace index with an eigenvalue-based characterization. As a consequence, we derive a characterization of $t(A)$ that reduces the computation of the zero-trace index of an arbitrary matrix to that of Hermitian matrices (Corollary \ref{t(a) reduce}). 
	In Section 5, we determine the zero-trace index for normal matrices, KMS matrices, weighted permutation matrices, companion matrices, $S_n$-matrices, and Jordan blocks.

\section{Preliminaries on $k$-numerical ranges}
	
	Let $1\leq k\leq n$. 
	Recall that the $k$-numerical range captures the trace behaviour of $k\times k$ compressions of a matrix.
	We begin by reviewing some properties of the $k$-numerical range  \cite{FillmoreWilliams1969,chikwong1994}.
	\begin{theorem}\label{prop of Wk}
	For $A\in M_n$, the following hold:
	\begin{enumerate}[(i)]
		\item $W_k(\alpha I_n +A)=k\alpha +W_k(A)$ for any $\alpha\in\mathbb{C}$; 
		\item $W_k(\beta A)=\beta W_k(A)$ for any $\beta\in \mathbb{C}$; 
		\item $W_k(\textup{Re}(A))=\textup{Re}(W_k(A))$; 
		\item $W_k(U^*AU)=W_k(A)$ for any unitary $U\in M_n$; 
		\item $W_{k}(A)=\textup{tr}(A)-W_{n-k}(A)$ for $1\leq k\leq n-1$;
		\item \label{Wn property}$W_n(A)=\{\textup{tr}(A)\}$; 
		\item $W_k(A)$ contains each sum $\lambda_{i_1}(A)+\dots+\lambda_{i_k}(A)$ of eigenvalues of $A$, where $\{i_1,\dots,i_k\}\subseteq \{1,\dots,n\}$ are distinct;
		\item \label{Wk for hermitian} If $A$ is Hermitian, then $\displaystyle W_k(A)=\left[\sum_{j=n-k+1}^n \!\!\lambda_j(A), \sum_{j=1}^k \lambda_j(A) \right]$;
		\item \label{nested k-num range} $\frac{1}{k+1} W_{k+1}(A)\subseteq \frac{1}{k} W_k(A)$ for $1\leq k\leq n-1$;
		\item \label{Wk direct sum} For $B\in M_m$ and $1\leq k\leq n+m $,  \\
	\(
	W_k(A\oplus B)=\textup{conv}\bigcup\{W_j(A)+W_{k-j}(B): 0\leq j\leq n, 0\leq k-j\leq m\},
	\)
	where $W_0(C):=\{0\}$ for any square matrix $C$.
	\end{enumerate}
	\end{theorem}
	
	An important step in expressing $d(A)$ as a minimum value is writing $\Lambda_k(A)$ as an intersection of certain half-planes (see \cite{LiSze2008}). 
	In the following theorem, we show that $W_k(A)$ can be characterized similarly. 
	This geometric description of $W_k(A)$ aligns with existing literature on \emph{$C$-numerical ranges} \cite{chikwong1994}. 
	Furthermore, this formulation allows us to express the zero-trace index as a minimum value, parallel to \cite[Theorem 2.2]{GauWangWu2014}. 
	As a result, we can reduce the computation of $t(A)$ for general complex matrices to the Hermitian case.

	\begin{theorem}\label{Wk as int}
	 Let $A\in M_n$ and $1\leq k\leq n$. Then
	\[ 
	W_k(A) = \bigcap_{\theta \in [0,\, 2\pi)} \left\{ \lambda \in \mathbb{C} : \mathrm{Re}(e^{i\theta} \lambda) \leq \sum_{j=1}^k \lambda_j(\mathrm{Re}(e^{i\theta} A)) \right\}.
	\]
	\end{theorem}
	\begin{proof}
	Let $w \in W_k(A)$. Fix $\theta \in [0,2\pi)$. 
	By Theorem \ref{prop of Wk} (ii) and (iii),
	\[
	\Re(e^{i\theta} w) \in \Re(e^{i\theta}(W_k(A))) 
	= \Re(W_k(e^{i\theta}A)) 
	= W_k(\Re(e^{i\theta}A)).
	\]
	Since $\Re(e^{i\theta}A)$ is Hermitian, Theorem \ref{prop of Wk} (viii) implies
	\[
	W_k(\Re(e^{i\theta}A)) 
	= \left[
	\sum_{j=n-k+1}^{n} \lambda_{j}(\Re(e^{i\theta}A)),
	\ \sum_{j=1}^{k} \lambda_j(\Re(e^{i\theta}A))
	\right].
	\]
	In particular,
	\[
	\Re(e^{i\theta} w) \le \sum_{j=1}^{k} \lambda_j(\Re(e^{i\theta}A)).
	\]
	Hence,
	\[
	W_k(A) \subseteq 
	\bigcap_{\theta \in [0,2\pi)} 
	\left\{
	v \in \mathbb{C} : 
	\Re(e^{i\theta} v) \le 
	\sum_{j=1}^{k} \lambda_j(\Re(e^{i\theta}A))
	\right\}.
	\]

	Conversely, by way of contraposition, suppose $w \notin W_k(A)$. Since $W_k(A)$ is a closed convex set \cite{Berger1963}, there exists $\varphi \in [0,2\pi)$ such that
	\[
	 \Re(e^{i\varphi} w) > \Re(e^{i\varphi} z)
	\quad \text{for all } z \in W_k(A).
	\]
	Let $\{x_1,\dots, x_k\}$ be an orthonormal set such that $x_j$ is an eigenvector of $\Re(e^{i\varphi}A)$ corresponding to 
	\(
	\lambda_j(\Re(e^{i\varphi}A)) \text{ for each } j \in \{1,\dots,k\}.
	\)
	Note that
	\(
	\sum_{j=1}^k x_j^* A x_j \in W_k(A),
	\)
	and
	\[
	\sum_{j=1}^k \lambda_j(\Re(e^{i\varphi}A))
	= \sum_{j=1}^k x_j^* \Re(e^{i\varphi}A)x_j 
	= \Re\left(e^{i\varphi} \sum_{j=1}^k x_j^* A x_j\right) 
	< \Re(e^{i\varphi} w).
	\]
	Thus,
	\[
	w \notin \left\{ v \in \mathbb{C} : \Re(e^{i\varphi} v) \le \sum_{j=1}^k \lambda_j(\Re(e^{i\varphi}A)) \right\}.
	\]
	It follows that
	\[
	w \notin \bigcap_{\theta \in [0,2\pi)} 
	\left\{ v \in \mathbb{C} : \Re(e^{i\theta} v) \le \sum_{j=1}^k \lambda_j(\Re(e^{i\theta}A)) \right\},
	\]
	and so the claim follows. 
	\end{proof}
	
\section{Basic properties of the zero-trace index}	

	We start by establishing basic properties of the zero-trace index. 
	Following the notations in \cite{GauWangWu2014}, let
	$i_{+}(C)$ (resp., $i_{-}(C)$ and $i_{0}(C)$) denote the number of positive (resp., negative and zero) eigenvalues of a Hermitian matrix $C$, counting multiplicities. Then $i_{\geq0}(C)=i_{+}(C)+i_{0}(C)$ is the number of nonnegative eigenvalues of $C$.
	
	\begin{remark}
	If $0\notin W(A)$, then $0\notin W_k(A)$ for all $1\leq k\leq n$ by Theorem \ref{prop of Wk} (ix). In this case, we set $t(A)=0$. Hereinafter, we assume $0\in W(A)$. 
	\end{remark}
	
	\begin{proposition} \label{t(A) properties}
	Let $A\in M_n$ and $B\in M_m$. The following hold:
	\begin{enumerate}[(i)]
	\item $d(A)\leq t(A)$;
	\item $t(A)=n$ if and only if $\textup{tr}(A)=0$;
	\item $t(A)=t(U^*AU)$ for any unitary $U\in M_n$;
	\item if $A$ is unitarily similar to $-A$, then $t(A)=n$;
	\item $t(A\oplus B)\geq t(A)+t(B)$.
	\end{enumerate}		
	\end{proposition}
	\begin{proof}
	Let $A\in M_n$ and $B\in M_m$.
	\begin{enumerate}[(i)]
	\item This follows from the definition of $t(A)$.
	\item Since $W_n(A)=\{\tr(A)\}$, it follows that $0\in W_n(A)$ if and only if $\tr(A)=0$.
	\item This is immediate from Theorem \ref{prop of Wk} (iv).
	\item If $A$ is unitarily similar to $-A$, then $\tr(A)=0$ and so $t(A)=n$.
	\item Since $0\in W_{t(A)}(A), W_{t(B)}(B)$, there exist $P\in M_{n,t(A)}$ and $Q\in M_{m,t(B)}$ such that $P^*P=I_{t(A)}$, $Q^*Q=I_{t(B)}$, and $\tr(P^*AP)=0=\tr(Q^*BQ)$.
	Consider $P\oplus Q\in M_{n+m, t(A)+t(B)}$ with $(P\oplus Q)^*(P\oplus Q)=P^*P\oplus Q^*Q=I_{t(A)}\oplus I_{t(B)}=I_{t(A)+t(B)}$. 
	Moreover, note that $\tr\left((P\oplus Q)^*(A\oplus B)(P\oplus Q)\right)$ can be simplified to
	\(
	\tr(P^*AP)+\tr(Q^*BQ)=0.
	\)
	Thus, by definition, $0\in W_{t(A)+t(B)}(A\oplus B)$, and so $t(A\oplus B)\geq t(A)+t(B)$. \qedhere
	\end{enumerate} 
	\end{proof}
	
	By Proposition \ref{t(A) properties} (i), $d(A)\leq t(A)$. We investigate when $d(A)=t(A)$. 
	Recall that a point $z$ in a convex set $S$ is an extreme point of $S$ if, whenever $z=\lambda x+(1-\lambda)y$ for some $x,y\in S$ and $\lambda\in(0,1)$, we have $x=y=z$ (see \cite{Watrous2018}).
	It was shown in \citep[Proposition 2.1]{GauWangWu2014} that if $0\in\partial W(A)$, then
	\[
	d(A)\leq \dim\text{span}\{x\in\mathbb{C}^n : x^*Ax=0\},
	\]
	with equality if and only if $0$ is an extreme point of $W(A)$. 
	The next theorem reveals a broader role for this geometric condition, showing that it serves as a sufficient condition for the equality of the two indices.

	\begin{theorem}\label{0 extreme point}
	Let $A\in M_n$. If  $0$ is an extreme point of $W(A)$, then $t(A)\leq \dim \textup{span}\{x\in \mathbb{C}^n : x^*Ax=0\}$ and $t(A)=d(A)$. 
	\end{theorem}
	\begin{proof}
	Let $t=t(A)$ and $K=\{x\in \mathbb{C}^n : x^*Ax=0\}$.
	Let $\{x_1,\dots,x_{t}\}\subset \mathbb{C}^n$ be orthonormal such that 
	$0=\sum_{j=1}^{t} x_j^*Ax_j\in W_{t}(A)$.
	Since $x_j^*Ax_j\in W(A)$ for each $1\leq j\leq t$, and $W(A)$ is convex, it follows that
	\[	
	0 = \frac{1}{t}\sum_{j=1}^{t} x_j^*Ax_j
	= \sum_{j=1}^{t} \frac{1}{t} x_j^*Ax_j
	\in W(A).
	\] 
	As $0$ is an extreme point of $W(A)$, we must have $x_j^*Ax_j=0$ for all $1\leq j\leq t$. 
	Thus, $\{x_1,\dots,x_t\}\subset K$ and so $t\leq \dim \text{span}\, K$. 
	By \cite[Proposition 2.1]{GauWangWu2014}, $d(A)=\dim \text{span}\, K$. We have $d(A)\leq t(A)\leq \dim \text{span}\, K=d(A)$, and hence $d(A)=t(A)$. 
	\end{proof}
	
	The converse of Theorem \ref{0 extreme point} is false.
	Take $A=\text{diag}(1,-2,-3)\in M_3$. 
	Then $0\in W(A)=[-3,1]$ and $0\notin W_2(A)=[-5,-1]$ by Theorem \ref{prop of Wk} (viii), so that $1\leq d(A)\leq t(A)= 1$, or $d(A)=t(A)=1$. 
	Moreover, 
	$t(A)<\dim \text{span}\{x\in \mathbb{C}^3 : x^*Ax=0\}=3$.	
	However, $0$ is not an extreme point of $W(A)$. \\
	
	We now show that when $A\in M_n$ is normal,  the converse of Theorem \ref{0 extreme point} holds precisely when $d(A)=t(A)=\dim\ker(A)$. 
	We also describe when the equality $d(A)=t(A)$ holds for a Hermitian matrix $A\in M_3$.

	\begin{corollary}
	Let $A\in M_n$ be normal. Then $0$ is an extreme point of $W(A)$ if and only if $d(A)=t(A)=\dim\ker(A)$.  
	\end{corollary}
	\begin{proof}
	Let $A\in M_n$ be normal.
	Set $t=t(A)$, $k=\dim\ker(A)$, and $K=\{x\in \mathbb{C}^n : x^*Ax=0\}$.
	If $A=0_n$, then the assertion follows.
	Suppose $A\neq 0_n$.
	If $0$ is an extreme point of $W(A)$, then $0$ is an eigenvalue of $A$ \cite{MacCluer1965} (see \cite[Section 1.6 Problem 17]{HornJohnson1991}), and $0<k<n$. 
	Clearly, $\ker(A)\subseteq K$.
	By the Spectral theorem, there exist unitary $U\in M_n$ and $a_1,\dots,a_{n}\in\mathbb{C}$ with $a_{n-k+1}=\dots=a_n=0$ such that $A=U^*\text{diag}(a_1,\dots,a_n)U$. 
	Let $x\in K$ and $Ux=[y_j]_{j=1}^n$. We have
	\[
	0=x^*Ax
	= x^*U^*\text{diag}(a_1,\dots,a_n)Ux
	= \sum_{j=1}^n |y_j|^2a_j
	= \sum_{j=1}^{n-k} |y_j|^2a_j.
	\]
	If $y_j\neq 0$ for some $1\leq j\leq n-k$, then we have $0=\sum_{j=1}^{n-k} \frac{|y_j|^2}{|y_1|^2+\dots+|y_{n-k}|^2} a_j$, that is, $0$ is a nontrivial convex combination of elements of $W(A)$. 
	This contradicts the assumption that $0$ is an extreme point of $W(A)$. Thus, $y_j=0$ for all $1\leq j\leq n-k$, and so 
	$Ax=U^*\text{diag}(a_1,\dots,a_n)(Ux)=U^*0_{n,1}=0_{n,1}$.
	Hence,  $x\in \ker(A)$ and $K\subseteq\text{ker}(A)$. 
	Thus, we have $\ker(A)=K$.
	Moreover, $\ker(A)=K$ implies $k\leq d(A)=t(A)\leq k$ or $d(A)=t(A)=\dim\ker(A)$ by \cite[Theorem 4.1]{GauWangWu2014} and Theorem \ref{0 extreme point}.
	Now, suppose $d(A)=t(A)=\dim\ker(A)$. 
	It follows from \cite[Theorem 4.1 (a)]{GauWangWu2014} that $0$ is not in the convex hull of the nonzero eigenvalues of $A$. 
	Hence, $0$ is an extreme point of $W(A)$ by \cite[Theorem 1.6.8]{HornJohnson1991}.
	\end{proof}

	\begin{proposition}\label{3x3 hermitian}
	Let $A\in M_3$ be Hermitian with eigenvalues $\lambda_1\geq \lambda_2\geq \lambda_3$. Then $d(A)=t(A)$ if and only if one of the following holds:
	\begin{enumerate}[(i)]
	\item $A=0_3$;
	\item $0$ is an eigenvalue of $A$ and $\textup{tr}(A)\neq 0$;
	\item $\lambda_1+\lambda_2<0$ or $\lambda_2+\lambda_3>0$.
	\end{enumerate}
	\end{proposition}
	\begin{proof}
	For $1\leq k\leq 3$, $W_k(A)=\left[\sum_{j=4-k}^3 \lambda_j, \sum_{j=1}^k \lambda_j\right]$. 
	Note that $d(A)=t(A)=3$ if and only if $A=0_3$. 
	If  $d(A)=t(A)=2$, then $0$ is an eigenvalue of $A$ by \citep[Lemma 3.1]{GauWangWu2014} and $\tr(A)\neq 0$ since $0\notin W_3(A)$.
	Conversely, suppose $0$ is an eigenvalue of $A$ and $\tr(A)\neq 0$. It follows from $\tr(A)\neq 0$ that $1\leq d(A)\leq t(A)\leq 2$ and $0$ has multiplicity at most $2$. 
	If $0$ is an eigenvalue of multiplicity 2, then $2\leq d(A)\leq t(A)\leq 2$ and so $d(A)=t(A)=2$. 
	If $0$ is an eigenvalue of multiplicity $1$, then $\left(i_{\geq 0}(A),i_{\geq 0}(-A)\right)\in \left\{(3,1), (1,3), (2,2)\right\}$. 
	If $\left(i_{\geq 0}(A),i_{\geq 0}(-A)\right)=(3,1)$, then $\lambda_1\geq \lambda_2>0=\lambda_3$ so that $\lambda_2+\lambda_3>0$. 
	Thus, $0\notin W_2(A)$ and $d(A)=t(A)=1$. 
	Similarly, if $\left(i_{\geq 0}(A),i_{\geq 0}(-A)\right)=(1,3)$, then $\lambda_1=0> \lambda_2\geq \lambda_3$ so that $\lambda_1+\lambda_2<0$. 
	Hence, $0\notin W_2(A)$ and $d(A)=t(A)=1$. 
	Lastly, if $\left(i_{\geq 0}(A),i_{\geq 0}(-A)\right)=(2,2)$, then $\lambda_1>0=\lambda_2>\lambda_3$. 
	By \citep[Theorem 2.4]{ChoiKribsZyczkowski2006}, $0\in \Lambda_2(A)=\{0\}$, and we have $d(A)=t(A)=2$. 
	Lastly, $d(A)=t(A)=1$ if and only if $0\notin W_2(A)$, that is, $\lambda_1+\lambda_2<0$ or $\lambda_2+\lambda_3>0$.
		\end{proof}
	
	Proposition \ref{t(A) properties} (ii) characterizes matrices that satisfy $t(A)=n$. 
	 We are interested in characterizing matrices $A\in M_n$ such that $t(A)=n-1$ (see \citep[Theorem 3.3]{GauWangWu2014} for the zero-dilation index analogue). 

\begin{theorem}\label{t n-1}
	Let $A\in M_n$. Then $t(A)=n-1$ if and only if $\textup{tr}(A)\neq 0$ and $\textup{tr}(A)\in W(A)$.
\end{theorem}
	\begin{proof}
	Note that $t(A)=n-1$ if and only if $0\in W_{n-1}(A)$ and $0\notin W_n(A)$. Equivalently, by Theorem \ref{prop of Wk} (v) and (vi), $0\in \tr(A)-W(A)$ and $\tr(A)\neq 0$, which holds if and only if  $\tr(A)\in W(A)$ and $\tr(A)\neq 0$.   
	\end{proof}

	Proposition \ref{t(A) properties} (v) gives a lower bound $t(A\oplus B)\geq t(A)+t(B)$, and equality does not always hold. Indeed, even for the case $A=B$, equality is not guaranteed.
	For example, take $A=\text{diag}(1,-2)$. 
	By Theorem \ref{prop of Wk} (\ref{Wk for hermitian}) and (ii), $0\in W(A)=[-2,1]$ and $t(A)=1$, respectively. 
	Consider $A\oplus A=\text{diag}(1,-2,1,-2)$, which is Hermitian, and we have $W(A\oplus A)=[-2,1],\  W_2(A\oplus A)=[-4,2],\ W_3(A\oplus A)=[-3,0]$, and $W_4(A\oplus A)=\{-2\}$. 
	By definition, $t(A\oplus A)=3>2t(A)$.\\
	
	In the following result, we give a necessary and sufficient condition for the equality $t(A\oplus A)=2t(A)$ to hold when $A\in M_n$ is Hermitian.
	
	\begin{proposition}\label{t(A plus A)}
	Let $A\in M_n$ be Hermitian with eigenvalues $\lambda_1\geq \cdots\geq \lambda_n$. Then $t(A\oplus A)=2t(A)$ if and only if one of the following holds:
	\begin{enumerate}[(i)]
	\item $\textup{tr}(A)=0$;
	\item $2\sum_{j=1}^{t(A)}\lambda_j + \lambda_{t(A)+1}< 0$ or $\,2\sum_{j=n-t(A)+1}^{n}\lambda_j + \lambda_{n-t(A)}> 0$.
	\end{enumerate}
	\end{proposition}
	\begin{proof}
	If $\tr(A)=0$, then $\tr(A\oplus A)=0$. It follows from Proposition \ref{t(A) properties} (ii) that $t(A\oplus A)=2n=2t(A)$.
	Suppose $\tr(A)\neq 0$. 
	Then $t(A\oplus A)=2t(A)$ if and only if $0\notin W_{2t(A)+1}(A\oplus A)$ by Proposition \ref{t(A) properties} (v).
	The latter condition is equivalent to $2\sum_{j=1}^{t(A)}\lambda_j + \lambda_{t(A)+1}< 0$ or $\,2\sum_{j=n-t(A)+1}^{n}\lambda_j + \lambda_{n-t(A)}> 0$ by Theorem \ref{prop of Wk} (\ref{Wk for hermitian}).
	\end{proof}
	
	It remains to be determined if a  characterization for the direct summands $A$ and $B$ exists for the equality to hold in Proposition \ref{t(A) properties} (v).
	In contrast, a necessary and sufficient  condition when $d(A\oplus B)=d(A)+d(B)$ is known (\citep[Corollary 2.8]{GauWangWu2014}).\\

	Although the zero-trace index is not additive with respect to direct sums in general, the following result shows that equality holds whenever either of the summands is the zero matrix. We make use of an equivalent characterization of the $k$-numerical range from \cite[Theorem 1.2]{FillmoreWilliams1969} as
	\begin{equation}\label{W_k as tr(AX)}
	W_k(A)=\left\{\tr(AP) : P\in M_n, 0\leq P\leq I_n, \tr(P)=k\right\},
	\end{equation}
	where, for $X,Y\in M_n$, $X\leq Y$ if $Y-X$ is positive semidefinite.\\
	
	\begin{proposition}\label{t(a+0)}
	Let $A\in M_n$. Then $t(A\oplus 0_m)=t(A)+m$.
	\end{proposition}
	\begin{proof}
	By Theorem \ref{prop of Wk} (\ref{Wk direct sum}), 
	\begin{align}\label{W_k A oplus 0m}
	W_k(A\oplus 0_m) &= \text{conv}\bigcup\{W_j(A)+W_{k-j}(0_m): 0\leq j\leq n,\ 0\leq k-j\leq m\}\nonumber\\
	&= \text{conv}\bigcup\{W_j(A)+0 : 0\leq j\leq n,\  0\leq k-j\leq m\}\nonumber\\
	&= \text{conv}\bigcup\left\{W_j(A) : \max\{0,k-m\}\leq j\leq \min\{k,n\}\right\}.
	\end{align}
	Let $t=t(A\oplus 0_m)$ and, by way of contradiction and Proposition \ref{prop of Wk} (v), suppose $t> t(A)+m$. 
	Then $t-m \geq  t(A)+1$ and $0\in W_t(A\oplus 0_m)$. Let $r=\min\{t,n\}$. 
	It follows from (\ref{W_k A oplus 0m}) that $0=\sum_{j=t-m}^r \alpha_j w_j$, where  $\alpha_j\geq 0$ and $w_j\in W_j(A)$ for each $t-m\leq j\leq r$  and $\sum_{j=t-m}^r \alpha_j=1$.
	For each $t-m\leq j\leq r$, $w_j=\tr(AP_j)$ where $P_j\in M_n$ such that $0\leq P_j\leq I_n$ and $\tr(P_j)=j$ by (\ref{W_k as tr(AX)}).
	Consider $P=\sum_{j=t-m}^r \alpha_j P_j\in M_n$. 
	For any unit vector $x\in \mathbb{C}^n$, note that $x^*Px=\sum_{j=t-m}^r \alpha_j x^*P_jx \in [0,1]$ since each $P_j$ is positive semidefinite with eigenvalues at most 1. Hence, $0\leq P\leq I_n$, and
	\begin{equation}\label{tr(P)}
	\tr(P)=\sum_{j=t-m}^r \alpha_j \tr(P_j) = \sum_{j=t-m}^r \alpha_j j \geq \sum_{j=t-m}^r \alpha_j (t(A)+1) = t(A)+1.
	\end{equation}
	Let $Q=\frac{\lfloor \tr(P)\rfloor}{\tr(P)}P$. Since $0<\frac{\lfloor \tr(P)\rfloor}{\tr(P)}\leq 1$, it follows that $0\leq Q\leq I_n$ and $\tr(Q)=\frac{\lfloor \tr(P)\rfloor}{\tr(P)}\tr(P)= \lfloor \tr(P)\rfloor \geq t(A)+1$ by (\ref{tr(P)}).
	Moreover, 
	\begin{align*}
	0 = \frac{\lfloor \tr(P)\rfloor}{\tr(P)}\sum_{j=t-m}^r \alpha_j w_j
	&= \frac{\lfloor \tr(P)\rfloor}{\tr(P)}\sum_{j=t-m}^r \alpha_j \tr(AP_j)\\
	&= \frac{\lfloor \tr(P)\rfloor}{\tr(P)}\tr\left(\sum_{j=t-m}^r \alpha_j AP_j\right)\\
	&= \tr\left(A \left(\frac{\lfloor \tr(P)\rfloor}{\tr(P)}\sum_{j=t-m}^r \alpha_jP_j\right)\right)\\
	&= \tr(AQ)\\
	&\in W_{\tr(Q)}(A),
	\end{align*}
	which contradicts the maximality of $t(A)$. Thus, $t(A\oplus 0_m)= t(A)+m$.
	\end{proof}
	
	
\section{Zero-trace index as a minimum value}

	For a Hermitian matrix $C\in M_n$, let $\ell_{+}(C)$ (resp., $\ell_{-}(C)$ and $\ell_{0}(C)$) be the number of positive (resp., negative and zero) partial sums of the form $\sum_{j=1}^{k}\lambda_j(C)$, for $1\leq k\leq n$, counting multiplicities. 
	We further define
	\[
	\ell_{\geq0}(C)=\ell_{+}(C)+\ell_{0}(C)
	\quad\text{and}\quad
	\ell_{\leq0}(C)=\ell_{-}(C)+\ell_{0}(C),
	\]
	which count the nonnegative and nonpositive partial sums of the ordered eigenvalues of $C$, respectively.

	\begin{example}\label{example}
	Let
	\(
	C=\textup{diag}(0,1,-1,-1)\in M_4,
	\)
	whose eigenvalues satisfy
	\(
	\lambda_1(C)=1,\, \lambda_2(C)=0,\, \lambda_3(C)=-1,\, \lambda_4(C)=-1.
	\)
	Then
	\(
	i_{+}(C)=1, i_{0}(C)=1, i_{-}(C)=2,
	\)
	so that
	\(
	i_{\geq0}(C)=2.
	\)
	The partial sums of the ordered eigenvalues are
	\[
	\begin{array}{c|c|c|c|c}
	 & k=1 & k=2 & k=3 & k=4 \\
	 \hline
	 \sum_{j=1}^k \lambda_j(C) & 1 & 1 & 0 & -1 
	\end{array}.
	\]	
	Hence,
	\(
	\ell_{+}(C)=2,\, \ell_{-}(C)=1,\, \ell_{0}(C)=1,
	\)
	and consequently
	$\ell_{\geq0}(C)=3$ and $\ell_{\leq0}(C)=2$.\\

	Meanwhile, $-C=\textup{diag}(0,-1,1,1)$ with $\lambda_j(-C)=-\lambda_{5-j}$ for $1\leq j\leq 4$.
	We have that
	\(
	i_{+}(-C)=2,\, i_{0}(-C)=1,\, i_{-}(-C)=1.
	\)
	The partial sums of the ordered eigenvalues are
	\[
	\begin{array}{c|c|c|c|c}
	 & k=1 & k=2 & k=3 & k=4 \\
	 \hline
	 \sum_{j=1}^k \lambda_j(-C) & 1 & 2 & 2 & 1 
	\end{array}.
	\]	
	so that 
	\(
	\ell_{+}(-C)=4,\, \ell_{-}(-C)=0,\, \ell_{0}(-C)=0,
	\)
	and $\ell_{\geq 0}(-C)=4$.
	\end{example}

	\begin{remark}\label{Remark}
	Let $C\in M_n$ be Hermitian.
	\begin{enumerate}[(i)]
	\item Note that $\ell_+(C)+\ell_0(C)+\ell_-(C)=n$.
	\item Since $\lambda_1(C)\geq \cdots\geq  \lambda_{i_{\geq 0}}(C)\geq 0$, then
	\(
	\sum_{j=1}^{i_{\geq0}(C)}\lambda_j(C)\ge0,
	\)
	and so 
	\(
	\ell_{\geq0}(C)\ge i_{\geq0}(C).
	\)
	
	\item Since $\sum_{j=1}^n \lambda_j(C)=\tr(C)$, $\ell_0(C)\geq 1$ whenever $\tr(C)=0$. 
	Moreover, $\tr(C)=0$ implies $\ell_{\geq 0}(C)=n$.

	\item Although $i_{+}(-C)=i_{-}(C)$, $i_0(C)=i_0(-C)$, and $i_-(-C)=i_+(C)$, the analogous identities $\ell_{+}(-C)=\ell_{-}(C)$, $\ell_0(C)=\ell_0(-C)$, and $\ell_-(-C)=\ell_+(C)$ need not hold as illustrated in Example \ref{example}.
	In general, since $\lambda_k(-C)=\lambda_{n-k+1}(C)$, $\ell_+(-C)$ counts the number of negative tail sums $\sum_{j=n-k+1}^n \lambda_j(C)$, while $\ell_-(C)$ counts the number of negative head sums $\sum_{j=1}^k \lambda_j(C)$, which is why there is no simple equality between them.
	\end{enumerate}
	\end{remark}	
	
	In the next lemma, we note an elementary property of the partial sums of eigenvalues arranged in decreasing order.
	
	\begin{lemma}\label{ell_geq 0 lemma}
	Let $C\in M_n$ be Hermitian. If $\sum_{j=1}^k \lambda_j(C)<0$ for some $1\leq k\leq n$, then $\sum_{j=1}^m \lambda_j(C)<0$ for any $m$ such that $k\leq m\leq n$. In this case, $\ell_{\geq 0}(C)\leq k-1$.
	\end{lemma}
	\begin{proof}
	Suppose $\sum_{j=1}^k \lambda_j(C)<0$ for some $1\leq k\leq n$. 
	If $\sum_{j=1}^m \lambda_j(C)\geq 0$ for some $k< m\leq n$, then
	$0> \sum_{j=1}^k \lambda_j(C) \geq -\sum_{j=k+1}^m \lambda_j(C)$,
	which implies $\sum_{j=k+1}^m \lambda_j(C)>0$. 
	Then $\lambda_l(C)>0$ for some $k+1\leq l\leq m$, and so $\lambda_j(C)>0$ for all $1\leq j\leq k<l$. 
	Consequently, $\sum_{j=1}^k \lambda_j(C)>0$ which is a contradiction.
	By definition of $\ell_{\geq 0}(C)$, we have $\ell_{\geq 0}(C)\leq k-1$
	\end{proof}

	We now state the main result of this section, which gives a spectral characterization of the zero-trace index, expressing it as the minimum number of nonnegative partial sums of the ordered eigenvalues of $\Re(e^{i\theta}A)$.

	\begin{theorem} \label{t index min} 
	For any $A\in M_n$, we have 
	\[
	t(A)=\min\left\{\ell_{\geq 0}(\textup{Re}(e^{i\theta}A)):\theta\in[0,2\pi)\right\}.
	\]
	\end{theorem}	
	
	\begin{proof}
	Let $t=t(A)$ and $m=\min\{\ell_{\geq 0}(\Re(e^{i\theta}A)):\theta\in[0,2\pi)\}$. 
	If $t=n$, then $0\in W_n(A)$ by Proposition \ref{t(A) properties} (ii), and so $\sum_{j=1}^n \lambda_j(\Re(e^{i\theta}A))\geq 0$ for all $\theta\in[0,2\pi)$ by Theorem \ref{Wk as int}. It follows that $\ell_{\geq 0}(\Re(e^{i\theta}A))=n$ for all $\theta\in[0,2\pi)$, and thus $m=n=t(A)$. 
	Assume $t\leq n-1$.
	By definition of $t(A)$, $0\in W_t(A)$ but $0\notin W_{t+1}(A)$.
	Then there exists $\theta'\in[0,2\pi)$ such that $\sum_{j=1}^{t+1} \lambda_j(\Re(e^{i\theta'} A))<0$ by Theorem \ref{Wk as int}.	 
	Hence, $\ell_{\geq 0}(\Re(e^{i\theta'} A))\leq t$ by Lemma \ref{ell_geq 0 lemma}, and so $m\leq t$. 
	Now, by the minimality of $m$,  there exists $\theta_0\in [0,2\pi)$ such that 
	\[
	\sum_{j=1}^m \lambda_j(\Re(e^{i\theta_0} A))\geq 0 \text{ and } \sum_{j=1}^{m+1} \lambda_j(\Re(e^{i\theta_0} A)) < 0. 
	\]
	Indeed, if $\sum_{j=1}^{m+1} \lambda_j(\Re(e^{i\theta} A))\geq 0$ for all $\theta\in [0,2\pi)$, then $\ell_{\geq 0}(\Re(e^{i\theta}A))\geq m+1$. 
	Consequently, $\min\{\ell_{\geq 0}(\Re(e^{i\theta}A)):\theta\in[0,2\pi)\}\geq m+1$ which contradicts the choice of $m$.
	It follows from Theorem \ref{Wk as int} that $0\notin W_{m+1}(A)$, and so $t\leq m$.	
	Therefore, $m=t=t(A)$.
	\end{proof}

	Several corollaries follow from the preceding theorem. The following result provides an explicit description of $t(A)$ for Hermitian matrices in terms of only $\ell_{\geq 0}(A)$ and $\ell_{\geq 0}(-A)$ (see \cite[Corollary 2.3]{GauWangWu2014} for $d(A)$ analogue).

	\begin{corollary}\label{t(a) Hermitian}
	If $A\in M_n$ is Hermitian, then 
	\[
	t(A)=\min\left\{\ell_{\geq 0}(A),\, \ell_{\geq 0}(-A)\right\}.
	\]
	\end{corollary}
	\begin{proof}
	Let $A\in M_n$ be Hermitian. 
	Since $\Re(e^{i\theta}A)=\cos(\theta)A$ for all $\theta\in[0,2\pi)$, the eigenvalues of $\Re(e^{i\theta}A)$ are $\{\cos(\theta)\lambda_1(A),\dots,\cos(\theta)\lambda_n(A)\}$. 
	If $\theta=\frac{\pi}{2}$ or $\frac{3\pi}{2}$, then $\cos(\theta)=0$ and so all the eigenvalues of $\Re(e^{i\theta}A)$ are $0$. 
	Hence, $\ell_{+}(\Re(e^{i\theta}A))=0$, $\ell_{0}(\Re(e^{i\theta}A))=n$, and so $\ell_{\geq 0}(\Re(e^{i\theta}A))=n$. 
	If $\theta\in \left[\left.0,\frac{\pi}{2}\right.\right)\cup \left(\frac{3\pi}{2},2\pi\right)$, then $\cos(\theta)>0$. 
	It follows that $\lambda_j(\Re(e^{i\theta}A))=\cos(\theta)\lambda_j(A)$ for all $1\leq j\leq n$, and so for $1\leq k\leq n$,
	\[
	\sum_{j=1}^k \lambda_j(\Re(e^{i\theta}A))= \sum_{j=1}^k \cos(\theta)\lambda_j(A)=\cos(\theta)\sum_{j=1}^k \lambda_j(A).
	\]
	Thus, $\ell_{\geq 0}(\Re(e^{i\theta}A))=\ell_{\geq 0}(A)$. 
	Now, if $\theta\in\left(\frac{\pi}{2},\frac{3\pi}{2}\right)$, then $\cos(\theta)<0$ and $\lambda_j(\Re(e^{i\theta}A))=\cos(\theta)\lambda_{n-j+1}(A)= -\cos(\theta)\lambda_j(-A)$ for $1\leq j\leq n$. 
	For $1\leq k\leq n$, we have that
	\[
	\sum_{j=1}^k \lambda_j(\Re(e^{i\theta}A)) 
	= \sum_{j=1}^k \cos(\theta)\lambda_{n-j+1}(A) 
	= -\cos(\theta) \sum_{j=1}^k \lambda_j(-A).
	\]
	It follows that $\ell_{\geq 0}(\Re(e^{i\theta}A))=\ell_{\geq 0}(-A)$ for all $\theta\in\left(\frac{\pi}{2},\frac{3\pi}{2}\right)$.
	Hence, 
	\[
	\ell_{\geq 0}(\Re(e^{i\theta}A))=
	\begin{cases}
	\ell_{\geq 0}(A) \, &\text{ if } \theta\in \left[\left.0,\frac{\pi}{2}\right.\right)\cup \left(\frac{3\pi}{2},2\pi\right), \\
	n \, &\text{ if } \theta=\frac{\pi}{2} \text{ or } \frac{3\pi}{2}, \\
	\ell_{\geq 0}(-A) \, &\text{ if } \theta\in\left(\frac{\pi}{2},\frac{3\pi}{2}\right).
	\end{cases}
	\]
	By Theorem \ref{t index min}, 
	\begin{align*}
	t(A) &= \min\left\{\ell_{\geq 0}(\Re(e^{i\theta}A)):\theta\in[0, 2\pi)\right\}\\
	&= \min\{\ell_{\geq 0}(A),n,\ell_{\geq 0}(-A)\}\\
	&= \min\{\ell_{\geq 0}(A),\ell_{\geq 0}(-A)\}. \qedhere
	\end{align*} 
		\end{proof}
	
	In \cite[Corollary 2.3]{GauWangWu2014}, there is a second assertion that if $C\in M_n$ is Hermitian, then $i_0(C)\leq d(C)\leq \lfloor (n+i_0(C))/2 \rfloor$.
	This does not hold for $t(C)$ since $l_0(C)\neq l_0(-C)$. 
	For example, take $C=\text{diag}(0,1,-1,-1)$ from Example \ref{example}. Then $\lfloor \frac{n+i_0(C)}{2} \rfloor=\lfloor \frac{4+1}{2}\rfloor=2$. By \cite[Corollary 2.3]{GauWangWu2014}, $d(C)=2$. Whereas $t(C)=\min \{3,4\}=3>2$ by Corollary \ref{t(a) Hermitian}. \\
	
	The next result shows a characterization that reduces the computation of the zero-trace index of an arbitrary matrix to that of Hermitian matrices.

	\begin{corollary}\label{t(a) reduce}
	For any $A\in M_n$, we have 
	\[
	t(A)=\min\left\{t(\textup{Re}(e^{i\theta}A)) : \theta\in[0,2\pi)\right\}.
	\]
	\end{corollary}
	\begin{proof}
	Since $\Re(e^{i\theta}A)$ is Hermitian, it follows from Theorem \ref{t index min} and Corollary \ref{t(a) Hermitian} that
	\begin{align*}
	&\min\left\{t(\Re(e^{i\theta}A)) : \theta\in[0,2\pi)\right\}\\
	&= \min\left\{\min\{\ell_{\geq 0}(\Re(e^{i\theta}A)), \ell_{\geq 0}(-\Re(e^{i\theta}A))\} : \theta\in[0,2\pi)\right\}\\
	&= \min\left\{\min\{\ell_{\geq 0}(\Re(e^{i\theta}A)), \ell_{\geq 0}(\Re(e^{i(\theta+\pi)}A))\} : \theta\in[0,2\pi)\right\}\\
	&= \min\left\{\ell_{\geq 0}(\Re(e^{i\theta}A)) : \theta\in[0,2\pi)\right\}\\
	&= t(A). \qedhere
	\end{align*} 
	\end{proof}

	As a consequence, an alternative proof for Proposition \ref{t(a+0)} is given below.\\
	
	Let $C\in M_n$ be Hermitian. Observe that
	\[
	\ell_{+}(C\oplus 0_m) = 
	\begin{cases}
	\ell_{+}(C)+m & \text{ if } i_{+}(C)\geq 1,\\
	\ell_+(C)=0 & \text{ if } i_{+}(C)=0,
	\end{cases}
	\]
	and 
	\[
	\ell_{0}(C\oplus 0_m) = 
	\begin{cases}
	\ell_{0}(C) & \text{ if } i_{+}(C)\geq 1,\\
	\ell_0(C)+m & \text{ if } i_{+}(C)=0.
	\end{cases}
	\]
	So, $\ell_{\geq 0}(C\oplus 0_m)=\ell_{\geq 0}(C)+m$. By Corollary \ref{t(a) reduce}, we have
	\begin{align*}
	t(A\oplus 0_m) &= \min \left\{\ell_{\geq 0}(\Re(e^{i\theta}(A\oplus 0_m)))\,:\, \theta\in[0,2\pi)\right\}\\
	&= \min \left\{\ell_{\geq 0}(\Re(e^{i\theta}A)\oplus 0_m)\,:\, \theta\in[0,2\pi)\right\}\\
	&= \min \left\{\ell_{\geq 0}(\Re(e^{i\theta}A) +m\,:\, \theta\in[0,2\pi)\right\} \\
	&= \min \left\{\ell_{\geq 0}(\Re(e^{i\theta}A) \,:\, \theta\in[0,2\pi)\right\}+m\\
	&= t(A)+m. 
	\end{align*}

\section{Zero-trace index of classes of matrices}
	
	The zero-dilation index has been determined for several important classes of matrices.
	Gau, Wang, and Wu \citep{GauWangWu2014} proved that if $A\in M_n$ is normal  with $k=\dim \text{ker}(A)$, then 
	\[
	k\leq d(A)\leq \left\lfloor \frac{n+k}{2}\right\rfloor,
	\] 
	and characterized the matrices attaining  equality on both sides. They also showed that any weighted permutation matrix (with zero diagonals) satisfies $d(A)\leq \lceil \frac{n}{2}\rceil$.
	Gau and Wu later obtained analogous results for $S_n$- and companion matrices \cite{GauWu2016}. 
	Recall that for a monic polynomial \( p(z) = z^{n} + a_{1}z^{n-1} + \cdots + a_{n-1}z + a_{n},\) its associated \emph{companion matrix} is defined as
	\[
	\begin{bmatrix}
	0 & 1 & & & \\
	& 0 & 1 & & \\
	& & \ddots & \ddots & \\
	& & & 0 & 1 \\
	-a_{n} & -a_{n-1} & \cdots & -a_{2} & -a_{1}
	\end{bmatrix},
	\]
	while a matrix $A\in M_n$ is called an \emph{$S_{n}$-matrix} if it is a contraction 
	\(
	\|A\| := \max_{x \in \mathbb{C}^{n}\setminus 	\{0\}} \frac{\|Ax\|}{\|x\|} \leq 1,
	\)
	has all its eigenvalues in the open unit disc $\{z\in\mathbb{C} : |z|< 1\}$, and satisfies $\text{rank}(I_{n} - A^{*}A) = 1$. 
	Additionally, the zero-dilation index was computed \citep{GauWu2014a} for the class of \emph{Kac-Murdock-Szegö (KMS) matrices}, defined for $a\in\mathbb{C}$ and $n\geq 1$ as
	\[
	\begin{bmatrix}
		0 & a & a^2 & \cdots & a^{n-1}\\
		 & 0 & a & \ddots & \vdots\\
		 & & \ddots & \ddots & a^2\\
		 & & & \ddots & a\\
		0 & & & & 0
		\end{bmatrix}.
	\]
	
	Combining these results with those established in the preceding section, we obtain the corresponding zero-trace indices in the following proposition.
	
	\begin{proposition}
	Let $A\in M_n$.
	\begin{enumerate}[(i)]
	\item If $A$ is normal with $k=\dim \ker(A)$, then $k\leq d(A)\leq t(A)\leq n$. 
	Moreover, if $A\neq 0_n$ and $\lambda_1,\dots,\lambda_{n-k}$ are the nonzero eigenvalues of $A$, then $t(A)=k$ if and only if $d(A)=k$ if and only if $0\notin\textup{conv}\{\lambda_1,\dots,\lambda_{n-k}\}$. 
	
	\item If $A$ is a KMS matrix or a weighted permutation matrix (with zero diagonals), then $t(A)=n$.
	
	\item Let $A$ be a companion matrix associated with a polynomial $z^{n} + a_{1}z^{n-1} + \cdots + a_{n-1}z + a_{n}$. Then $t(A)=n$ if $a_1=0$, and $t(A)=n-1$ if $a_1\neq 0$.
	
	\item Let $A$ be an $S_n$-matrix such that $d(A)=\lceil \frac{n}{2} \rceil$. Then $t(A)=n$ if $n$ is odd, and $t(A)\geq n-1$ if $n$ is even. 
	
	\item Let $J_n(\lambda)$ be the $n\times n$ Jordan block with eigenvalue $\lambda\in\mathbb{C}$. Then for $n\geq 2$, 
	\[
	t(J_n(\lambda))=\begin{cases}
	n & \text{ if } \lambda=0, \\
	k & \text{ if } \frac{1}{k+1}\sum_{j=1}^{k+1} \cos\left(\frac{j\pi}{n+1}\right) < |\lambda|\leq \frac{1}{k}\sum_{j=1}^k \cos\left(\frac{j\pi}{n+1}\right),\\
	  & \text{ where }  1\leq k< n,\\
	0 & \text{ if } |\lambda|>\cos\left(\frac{\pi}{n+1}\right).
	\end{cases}	
	\]
	\end{enumerate}
	\end{proposition}
	
	\begin{proof}
	\begin{enumerate}[(i)]
	\item If $A=0_n$, then $d(A)=t(A)=n$. 
	Suppose $A\neq 0_n$ is normal and $\lambda_1,\dots,\lambda_{n-k}$ its nonzero eigenvalues . Then $A$ is unitarily similar to $\text{diag}(\lambda_1,\dots,\lambda_{n-k})\oplus 0_k$. 
	Suppose $d(A)=k$. This is equivalent to $0\notin\textup{conv}\{\lambda_1,\dots,\lambda_{n-k}\}$ by \cite[Theorem 4.1]{GauWangWu2014}. 
	Hence, $t(A)=t(\text{diag}(\lambda_1,\dots,\lambda_{n-k}))+k=k$ by Proposition \ref{t(a+0)}.
	Conversely, if $t(A)=k$, then $d(A)=k$ by \cite[Theorem 4.1]{GauWangWu2014} and Proposition \ref{t(A) properties} (i). 	
	
	\item A KMS matrix and a weighted permutation matrix with zero diagonals are of trace 0. Hence, $t(A)=n$ by Proposition \ref{t(A) properties} (ii).
	
	\item If $A$ is a companion matrix associated with a polynomial $z^{n} + a_{1}z^{n-1} + \cdots + a_{n-1}z + a_{n}$, then the Jordan block $J_{n-1}(0)$ is a compression of $A$ with trace 0. So $t(A)\geq n-1$. If $a_1=0$, then $\tr(A)=0$ and $t(A)=n$. Otherwise, if $a_1\neq 0$, then $t(A)=n-1$.	
	
	\item Suppose $d(A)=\lceil\frac{n}{2}\rceil$. If $n$ is odd, then $A$ is unitarily similar to $-A$ by \cite[Theorem 2.2]{GauWu2016}. Hence, $t(A)=n$ by Proposition \ref{t(A) properties} (iv). If $n$ is even, then $A$ has a compression $B\in M_{n-1}$ such that $B$ is unitarily similar to $-B$ by \cite[Theorem 2.3]{GauWu2016}. By Proposition \ref{t(A) properties} (iv), $t(A)\geq n-1$.
	
	\item If $\lambda=0$, then $\tr(J_n(0))=0$ and so $t(J_n(0))=n$ by Proposition \ref{t(A) properties} (ii).
	Suppose $\lambda\neq 0$. Note that 
	$W_k(J_n(\lambda))=W_k(\lambda I_n+J_n(0))=k\lambda+W_k(J_n(0))$ by Theorem \ref{prop of Wk} (i), where $W_k(J_n(0))$ is the circular disk centered at the origin with radius equal to $\sum_{j=1}^k \lambda_j(\Re( J_n(0)))$. The eigenvalues of $\Re(J_n(0))$ are $\lambda_j(\Re( J_n(0)))=\cos\left(\frac{j\pi}{n+1}\right)$ for $1\leq j\leq n$ \cite{GauWu2014a}. Thus,
	\begin{align*}
	0\in W_k(J_n(\lambda)) &\iff 0\in k\lambda + W_k(J_n(0))\\
	&\iff 0\in k\lambda+\left\{z\in \mathbb{C}:  |z|\leq \sum_{j=1}^k \cos\left(\frac{j\pi}{n+1}\right)\right\}\\
	&\iff -k\lambda \in \left\{z\in \mathbb{C}:  |z|\leq \sum_{j=1}^k \cos\left(\frac{j\pi}{n+1}\right)\right\}\\
	&\iff |\lambda|\leq \frac{1}{k}\sum_{j=1}^k \cos\left(\frac{j\pi}{n+1}\right).
	\end{align*}
	Note that $\frac{1}{k}\sum_{j=1}^k \cos\left(\frac{j\pi}{n+1}\right) > \frac{1}{k+1}\sum_{j=1}^{k+1} \cos\left(\frac{j\pi}{n+1}\right)$ for $1\leq k\leq n$. 
	Indeed, 
	\begin{align*}
	0 &< \frac{1}{k+1}\left[\cos\left(\frac{k\pi}{n+1}\right)-\cos\left(\frac{(k+1)\pi}{n+1}\right)\right]\\
	&\leq \frac{1}{k+1}\left[\frac{1}{k}\sum_{j=1}^k \cos\left(\frac{j\pi}{n+1}\right)-\cos\left(\frac{(k+1)\pi}{n+1}\right)\right]\\
	&= \left(\frac{1}{k}-\frac{1}{k+1}\right)\sum_{j=1}^k \cos\left(\frac{j\pi}{n+1}\right)-\frac{1}{k+1}\cos\left(\frac{(k+1)\pi}{n+1}\right)\\
	&= \frac{1}{k}\sum_{j=1}^k \cos\left(\frac{j\pi}{n+1}\right)- \frac{1}{k+1}\sum_{j=1}^{k+1} \cos\left(\frac{j\pi}{n+1}\right).
	\end{align*}
	Thus, if $|\lambda|>\cos\left(\frac{\pi}{n+1}\right)$, then $0\notin W(J_n(\lambda))$ and so $t(J_n(\lambda))=0$. If $|\lambda|\leq \cos\left(\frac{\pi}{n+1}\right)$, then $t(J_n(\lambda))=k$ if and only if $\frac{1}{k+1}\sum_{j=1}^{k+1} \cos\left(\frac{j\pi}{n+1}\right) < |\lambda|\leq \frac{1}{k}\sum_{j=1}^k \cos\left(\frac{j\pi}{n+1}\right)$. \qedhere
	\end{enumerate}
	\end{proof}
	
\section*{Declaration of Competing Interest}
There is no competing interest.

\section*{Acknowledgements}
The authors acknowledge the Office of the Chancellor of the University of the Philippines Diliman, through the Office of the Vice Chancellor for Research and Development, for funding support through the Outright Research Grant (262604 ORG).



\end{document}